\documentclass{amsart}
\usepackage{amssymb,pstricks}
\usepackage{latexsym}
\usepackage{amsmath}
\usepackage{amsthm}
\topskip  \newcommand{\des}{\mathrm{des}}
\newcommand{\rr}{\mathrm{r}}
\newcommand{\dd}{\mathrm{d}}
\newcommand{\pp}{\mathrm{p}}
\newcommand{\vv}{\mathrm{v}}
\newcommand{\s}{\mathrm{st}}

\newtheorem{theorem}{Theorem}[section]
\newtheorem{corollary}[theorem]{Corollary}
\newtheorem{lemma}[theorem]{Lemma}
\newtheorem{proposition}[theorem]{Proposition}

\numberwithin{equation}{section}

\title{Counting subwords in flattened permutations}

\author{Toufik Mansour}
\address{Department of Mathematics, University of Haifa, 31905 Haifa, Israel}
\email{tmansour@univ.haifa.ac.il}

\author{Mark Shattuck}
\address{Department of Mathematics, University of Tennessee, Knoxville, TN 37996}
\email{shattuck@math.utk.edu}

\author{David G.L. Wang}
\address{Department of Mathematics, University of Haifa, 31905 Haifa, Israel}
\email{david.combin@gmail.com, wgl@math.haifa.ac.il}

\subjclass[2010]{05A15, 05A05}

\keywords{pattern avoidance; permutation; kernel method}

\begin{document}

\begin{abstract}
In this paper, we consider the number of occurrences of descents, ascents, $123$-subwords, $321$-subwords,
peaks and valleys in flattened permutations, which were recently introduced by Callan in his study of finite set partitions.
For descents and ascents, we make use of the kernel method and obtain an explicit formula (in terms of Eulerian polynomials) for the distribution on $\mathcal{S}_n$ in the flattened sense.
For the other four patterns in question,
we develop a unified approach to obtain explicit formulas for the comparable distributions.
We find that the formulas so obtained for $123$- and $321$-subwords can be expressed in terms of the Chebyshev polynomials of the second kind,
while those for peaks and valleys are more related to the Eulerian polynomials.  We also provide a bijection showing the equidistribution of descents in flattened permutations of a given length with big descents in permutations of the same length in the usual sense.
\end{abstract}

\maketitle
%\tableofcontents

\section{Introduction}

This paper concerns the enumeration problem for particular subwords of length two or three.
The subword counting problem for permutations has been studied extensively from various perspectives in both enumerative and algebraic combinatorics; see, e.g., \cite{EN, Ki}.
The comparable problem has also been considered on other discrete structures such as $k$-ary words \cite{BM}, compositions \cite{MSi}, and set partitions \cite {MSY} (see also \cite{HM} and the references contained therein).
Here, we consider a variant of the subword problem using the concept of a flattened permutation introduced recently
by Callan~\cite{C} in his study of set partitions (and considered further by two of the present authors~\cite{MS}).

Let $\pi$ be a permutation of length $n$ represented in \emph{standard cycle form}, i.e., cycles arranged from left to right in ascending order according to the size of the smallest elements, where the smallest element is first within each cycle.
Define $\text{Flatten}(\pi)$ to be the permutation of length $n$
obtained by erasing the parentheses
enclosing the cycles of $\pi$ and considering the resulting word.
For example, if $\pi=71564328 \in \mathcal{S}_8$,
then the standard cycle form of $\pi$ is $(172)(3546)(8)$ and $\text{Flatten}(\pi)=17235468$.  Throughout, $\mathcal{S}_n$ will denote the set of permutations of size $n$.

Let $\pi=\pi_1\pi_2\cdots\pi_n$ and $\sigma=\sigma_1\sigma_2\cdots\sigma_d$ be permutations of length $n$ and $d$, where $n \geq d$.
In the typical setting, the permutation $\pi$ is said to \emph{contain} $\sigma$ as a \emph{subword} if there exists a set of consecutive letters $\pi_i\pi_{i+1}\cdots\pi_{i+d-1}$ in $\pi$ that is order-isomorphic to $\sigma$. Otherwise, $\pi$ is said to \emph{avoid} $\sigma$.  In this context, $\sigma$ is usually called a (subword) \emph{pattern}.  For example, the permutation $\pi=1247653 \in \mathcal{S}_7$ (represented as a \emph{word}) contains two occurrences of the pattern $321$ (corresponding to $765$ and $653$; note that occurrences of a given pattern need not be disjoint), but avoids the pattern  $213$.    Subwords of the form $\pi_i\pi_{i+1}$, where $1\le i\le n-1$ and $\pi_i>\pi_{i+1}$ (resp., $\pi_i<\pi_{i+1}$), are called \emph{descents} (resp., \emph{ascents}).
A set of letters $\pi_i\pi_{i+1}\pi_{i+2}$, where $1\le i\le n-2$,
is said to be a {\em $123$-subword} if $\pi_i<\pi_{i+1}<\pi_{i+2}$,
a {\em $321$-subword} if $\pi_i>\pi_{i+1}>\pi_{i+2}$,
a {\em peak} if $\pi_{i+1}=\max\{\pi_i,\pi_{i+1},\pi_{i+2}\}$, or a {\em valley} if $\pi_{i+1}=\min\{\pi_i,\pi_{i+1},\pi_{i+2}\}$.

In this paper, we will consider an alternative definition of subword containment for permutations obtained by looking at the comparable problem on $\text{Flatten}(\pi)$.  More specifically, we will say from now on that a permutation~$\pi$ contains an occurrence of the subword pattern~$\rho$ in the flattened sense if and only if
$\text{Flatten}(\pi)$ contains the subword $\rho$ in the usual sense (and say~$\pi$ avoids~$\rho$ otherwise).
In the current paper, using this new definition,
we will concentrate on the following three pairs of subword patterns, namely,
descents and ascents, $123$- and $321$-subwords, and peaks and valleys.
We provide a unified approach below for dealing with these patterns and determine in each case an explicit formula for the distribution of the pattern on $\mathcal{S}_n$ in the flattened sense.  Our results may often be expressed in terms of either Eulerian or Chebyshev polynomials.
In some cases, formulas for the number of permutations having a fixed number of occurrences of the pattern in question are also given.  In all cases, we give simple formulas for the average number of occurrences of each pattern, providing both algebraic and combinatorial proofs.

In what follows, we will let $\s$ denote a statistic defined on flattened permutations.
Let $g^\s_n$ denote the polynomial obtained by considering the distribution of the statistic $\s=\s(\text{Flatten}(\pi))$ taken over all permutations $\pi$ of length~$n$, that is
$$g^\s_n=\sum_{\pi \in \mathcal{S}_n}q^{\s(\text{Flatten}(\pi))}.$$
Define the generating function
\[
G^\s(x)=\sum_{n\ge0}g_{n+1}^\s{x^n\over n!}.
\]

In the next section, we determine $G^\s(x)$ for the statistics recording the number of descents and ascents in $\text{Flatten}(\pi)$, making use of the \emph{kernel method} \cite{BBD}, from which one may deduce a formula for $g^\s_n$ in these cases.  In the third section, we consider the comparable question for statistics recording the number of $123$-subwords, $321$-subwords, peaks, and valleys.  We remark that in the case of descents, there are further combinatorial results.  First, it turns out that the descents statistic on permutations of a given length, taken in the flattened sense, has the same distribution as does the statistic on permutations of the same length for the number of descents of size two or more, taken in the usual sense.  We provide a combinatorial proof of this fact by defining a suitable bijection of $\mathcal{S}_n$.  Furthermore, we also consider descents of size $d$ or more on flattened permutations, where $d \geq 1$, and provide a combinatorial proof of an explicit formula for the number of permutations of a given length having a fixed number of such descents.

We will use the following notation throughout this paper.
If $n$ is a positive integer, then let $[n]=\{1,2,\ldots,n\}$, with $[0]=\emptyset$.  If $m$ and $n$ are positive integers, then let $[m,n]=\{m,m+1,\ldots,n\}$ if $m \leq n$, with $[m,n]=\emptyset$ if $m>n$.  Define the characteristic function~$\chi$  by
\[
\chi(P)=\begin{cases}
1,&\text{if $P$ is true};\\
0,&\text{if $P$ is false},
\end{cases}
\]
for any proposition $P$.  Throughout, we let $\theta=1-q$, where $q$ is an indeterminate.

If $n \geq 0$,
then the {\em Eulerian polynomial} $A_n(q)$ is defined by
\[
A_n(q)=\sum_{k=0}^nA_{n,k}q^k
=(1-q)^{n+1}\sum_{j\ge1}j^nq^{j-1},
\]
where $A_{n,k}$ denotes the {\em Eulerian number} which counts the permutations of length~$n$ having exactly $k$ ascents (in the usual sense).
Recall that the generating function of $A_n(q)$ is given by
\[
A(x,q)=\sum_{n\ge0}A_n(q){x^n\over n!}
={1-q\over e^{(q-1)x}-q};
\]
see, for example, Graham et al.~\cite[p. 351]{GKP94}.  (See also Hirzebruch~\cite{Hir08} and Foata~\cite{Foa10} for further information on the Eulerian numbers and polynomials.)
The Eulerian numbers are sometimes defined as
$E_{n,k}=A_{n,k-1}$ when it is more convenient.  Note
the generating function
\begin{equation}\label{gf:E}
E(x,q)=\sum_{n\geq1}\sum_{k\ge1}E_{n,k}q^k\frac{x^n}{n!}
=\frac{q(e^{qx}-e^x)}{qe^x-e^{qx}},
\end{equation}
which will be used in the third section below.

Finally, recall that the Chebyshev polynomials $U_n(x)$ of the second kind (see Rivlin~\cite{R})
are defined by the initial values $U_{-2}(t)=-1$ and $U_{-1}(t)=0$, along with the recurrence
\begin{equation}\label{def:ChebyshevPolynomial}
U_{n}(t)=2t\cdot U_{n-1}(t)-U_{n-2}(t), \qquad n \geq 0.
\end{equation}
It is well-known that
\begin{equation}\label{fm:U}
U_n(t)={(t+\sqrt{t^2-1})^{n+1}-(t-\sqrt{t^2-1})^{n+1}\over2\sqrt{t^2-1}}.
\end{equation}

\section{Counting descents}\label{sec:des}

\subsection{Flattened descents}
We will use a more explicit notation to denote the generating functions under consideration in this section.  Let $F_\rho(n;q|a_1a_2\cdots a_k)$ be the generating function which counts permutations $\pi$ of length $n$ according to the number of occurrences of the subword $\rho$ in $\text{Flatten}(\pi)$ such that the first $k$ elements of the first cycle in the standard form of $\pi$ are $a_1a_2\cdots a_k$. Clearly, $F_{21}(n;q)=F_{21}(n;q|1)$, for all $n\geq1$. Considering whether or not the first cycle contains exactly one element yields
\begin{align}
F_{21}(n;q|1)&=F_{21}(n-1;q)+\sum_{j=2}^n F_{21}(n;q|1j).\label{eqdes1}
\end{align}
From the definitions, we have
\begin{align}
F_{21}(n;q|12)=F_{21}(n-1;q|1)=F_{21}(n-1;q).\label{eqdes2}
\end{align}
Considering whether the first cycle contains two or more elements yields for $j \geq 3$,
\begin{align*}
F_{21}(n;q|1j)&=qF_{21}(n-2;q)+\sum_{k=2}^{j-1}F_{21}(n;q|1jk)+\sum_{k=j+1}^nF_{21}(n;q|1jk)\\
&=qF_{21}(n-2;q)+q\sum_{k=2}^{j-1}F_{21}(n-1;q|1k)+\sum_{k=j}^{n-1}F_{21}(n-1;q|1k)\\
&=qF_{21}(n-2;q)+(q-1)\sum_{k=2}^{j-1}F_{21}(n-1;q|1k)+\sum_{k=2}^{n-1}F_{21}(n-1;q|1k),\\
\end{align*}
which, by \eqref{eqdes1}, implies
\begin{align}
F_{21}(n;q|1j)&=F_{21}(n-1;q)+(q-1)F_{21}(n-2;q)+(q-1)\sum_{k=2}^{j-1}F_{21}(n-1;q|1k), \qquad j \geq 3.\label{eqdes3}
\end{align}

If $n \geq 2$, then let $F_{21}(n;q,v)=\sum_{j=2}^nF_{21}(n;q|1j)v^{j-2}$. Multiplying \eqref{eqdes3} by $v^{j-2}$ and summing over $j=3,4,\ldots,n$, we obtain
\begin{align*}
F_{21}(n;q,v)-F_{21}(n;q|12)&=\frac{v-v^{n-1}}{1-v}F_{21}(n-1;q)+(q-1)\frac{v-v^{n-1}}{1-v}F_{21}(n-2;q)\\
&\quad+\frac{q-1}{1-v}
\sum_{k=2}^{n-1}\Bigl(F_{21}(n-1;q|1k)v^{k-1}-F_{21}(n-1;q|1k)v^{n-1}\Bigr).
\end{align*}
By \eqref{eqdes2}, this may be rewritten as
\begin{align}
F_{21}(n;q,v)&=\frac{1-v^{n-1}}{1-v}F_{21}(n-1;q)+(q-1)\frac{v-v^{n-1}}{1-v}F_{21}(n-2;q)\notag\\
&\quad+\frac{q-1}{1-v}\Bigl(vF_{21}(n-1;q,v)-F_{21}(n-1;q,1)v^{n-1}\Bigr), \qquad n \geq 2,\label{eqdes4}
\end{align}
with $F_{21}(1;q,v)=0$, $F_{21}(2;q,v)=1$, $F_{21}(1;q)=1$ and $F_{21}(2;q)=2$.

Let $H_{21}(t;q,v)=\sum_{n\geq1}F_{21}(n;q,v)t^n$ and $H_{21}(t;q)=\sum_{n\geq1}F_{21}(n;q)t^n$. Multiplying \eqref{eqdes4} by $t^n$ and summing over $n\geq2$, we obtain
\begin{align*}
H_{21}(t;q,v)
&=\frac{t}{1-v}(H_{21}(t;q)-H_{21}(vt;q))+\frac{(q-1)vt^2}{1-v}(H_{21}(t;q)-H_{21}(vt;q))\\
&\quad +\frac{(q-1)t}{1-v}(vH_{21}(t;q,v)-H_{21}(vt;q,1)).
\end{align*}
Note that by virtue of \eqref{eqdes1},
$$H_{21}(t;q,1)=\sum_{n\geq2}\sum_{j=2}^nF_{21}(n;q|1j)t^n=(1-t)H_{21}(t;q)-t.$$
Hence,
\begin{align*}
H_{21}(t;q,v)&=\frac{t}{1-v}(H_{21}(t;q)-H_{21}(vt;q))+\frac{(q-1)vt^2}{1-v}(H_{21}(t;q)-H_{21}(vt;q))\\
&+\frac{(q-1)t}{1-v}(vH_{21}(t;q,v)-(1-vt)H_{21}(vt;q)+vt).
\end{align*}
To solve this functional equation, we make use of the \emph{kernel method} (see \cite{BBD}). Comparing the coefficients of $H_{21}(t;q,v)$ on both sides and solving for $v=v_0$ in terms of $t$ and $q$, we get
$$v_0=\frac{1}{1-(1-q)t}.$$
Setting $v=v_0$ in the above equation implies
\begin{align}
H_{21}(t;q)=\frac{(1-q)t}{1-2(1-q)t}+\frac{q(1-(1-q)t)}{1-2(1-q)t}H_{21}\left(\frac{t}{1-(1-q)t};q\right).\label{eqdes5}
\end{align}
Iterating \eqref{eqdes5} (assuming $|t|$, $|q|<1$) gives
\begin{align*}
H_{21}(t;q)
&=\sum_{j=1}^m
\frac{q^{j-1}(1-\theta t)\theta t}{(1-j\theta t)\bigl(1-(j+1)\theta t\bigr)}
+\frac{q^m(1-\theta t)}{1-(m+1)\theta t}H_{21}\left(\frac{t}{1-m\theta t};q\right)
\end{align*}
for any $m\ge1$.
Letting $m\to\infty$ in the last expression, we see that the second summand tends to zero, which implies
\begin{align*}
H_{21}(t;q)
&=(1-\theta t)
\sum_{j\ge1}q^{j-1}
\biggl({1\over 1-(j+1)\theta t}-{1\over 1-j\theta t}\biggr)
=\theta(1-\theta t)\sum_{j\ge2}{q^{j-2}\over 1-j\theta t}-1.
\end{align*}
Extracting the coefficient of $t^n$ in the last expression yields the following result.

\begin{theorem}\label{thdes}
For any $n\ge1$, we have
\[
F_{21}(n;q)
=(1-q)^{n+1}\sum_{j\geq1}(j-1)j^{n-1}q^{j-2}
={1\over q}\Bigl(A_n(q)+(q-1)A_{n-1}(q)\Bigr),
\]
where $A_n(q)$ is the $n$-th Eulerian polynomial.
Moreover,
\begin{equation}\label{gf:des}
\sum_{n\ge0}F_{21}(n+1;q)\frac{x^n}{n!}
=\frac{(1-q)^2}{(e^{(q-1)x}-q)^2}=A(x,q)^2,
\end{equation}
where $A(x,q)$ is the generating function of the Eulerian polynomials.
\end{theorem}

\begin{corollary}\label{descor1}
The average number of descents in $\text{Flatten}(\pi)$ over $\pi\in S_n$ is given by $\frac{(n-1)(n-2)}{2n}$.
\end{corollary}
\begin{proof}
By differentiating the generating function $\sum_{n\geq1}F_{21}(n;q)\frac{x^{n-1}}{(n-1)!}$ in the statement of Theorem \ref{thdes} with respect to $q$ and taking the limit at $q=1$, we obtain
$$\sum_{n\geq1}\frac{d}{dq}F_{21}(n;q)\mid_{q=1}\frac{x^{n-1}}{(n-1)!}=\frac{x^2}{(1-x)^3},$$
which implies
$$\frac{1}{n!}\frac{d}{dq}F_{21}(n;q)\mid_{q=1}=\frac{1}{n}\binom{n-1}{2}=\frac{(n-1)(n-2)}{2n},$$
as required.
\end{proof}

Let us refer to a descent or ascent occurring within $\text{Flatten}(\sigma)$ as a \emph{flattened descent} or \emph{ascent}, respectively, of a permutation $\sigma$.
Since each of the $n-1$ adjacencies within any member of $\mathcal{S}_n$ is either a flattened descent or ascent, the following corollary is immediate from Theorem \ref{thdes}.

\begin{corollary}
For any $n\geq1$, we have
$$F_{12}(n;q)
=q^{n-1}F_{21}(n;q^{-1})
=(q-1)^{n+1}\sum_{j\geq1}(j-1)j^{n-1}q^{-j}.$$
\end{corollary}

A proof similar to before then gives the following result.

\begin{corollary}\label{asccor}
The average number of ascents in $\text{Flatten}(\pi)$ over $\pi\in S_n$ is given by $\frac{(n-1)(n+2)}{2n}$.
\end{corollary}

If $\pi_1\pi_2\cdots\pi_n$ is a permutation of length~$n$,
then we will call a position~$i\in[n-1]$ a {\em big descent} if $\pi_i-\pi_{i+1}\ge2$.

\begin{theorem}\label{thdes2}
The flattened descent and big descent statistics on $\mathcal{S}_n$ are equally distributed for all $n \geq 1$.
\end{theorem}
\begin{proof}
We will define a Foata style bijection showing this equivalence for all $n$.  To do so, first let $\mathcal{A}_n$, where $n \geq 2$, be the set comprising all sequences of length $n-1$ of ordered pairs $(a_i,b_i)$, $i \in [n-1]$, satisfying the following properties: (i) $a_i \in \{0,1\}$ for all $i$; (ii) $(a_1,b_1)=(0,1)$ or $(0,2)$; (iii) If $i \geq 2$ and $a_i=0$, then let $b_i$ be any member of $[s+2]$, where $s$ denotes the number of $1$'s occurring among the first coordinates of the ordered pairs $(a_1,b_1),(a_2,b_2),\ldots,(a_{i-1},b_{i-1})$; and (iv) If $i \geq 2$ and $a_i=1$, then let $b_i$ be any member of $[i-1-s]$, where $s$ is as in part (iii).  Note that $\mathcal{A}_n=n!$ since there are $i+1$ choices for the $i$-th ordered pair, $1 \leq i \leq n-1$.

We now define a bijection $g$ between $\mathcal{A}_n$ and $\mathcal{S}_n$ such that the number of $1$'s among the first coordinates of $\pi \in \mathcal{A}_n$ corresponds to the number of flattened descents within $g(\pi)$.  To create $g(\pi)$ from $$\pi=\{(a_1,b_1),(a_2,b_2),\ldots,(a_{n-1},b_{n-1})\}\in \mathcal{A}_n,$$ we first write $1$ in a cycle by itself and call this permutation $\sigma_1$.  We then subsequently add members of $[n]-\{1\}$ using $\pi$ as an encoding and create a recursive sequence of permutations $\sigma_2, \sigma_3,\ldots, \sigma_n$ as follows.  If $a_j=0$, where $j \in [n-1]$, then insert $j+1$ either within a cycle of $\sigma_{j}$ so that it goes in between any two letters comprising a descent of $\text{Flatten}(\sigma_j)$ (with the $\ell$-th such descent from left-to-right selected for the insertion site if $b_j=\ell \in [s]\subseteq[s+2]$, assuming that there are $s$ $1$'s appearing as first coordinates within the first $j-1$ ordered pairs of $\pi
 $, with $s=0$ if $j=1$) or at the end of the last cycle of $\sigma_j$ (if $b_j=s+1$) or as the $1$-cycle $(j+1)$ (if $b_j=s+2$).  If $a_j=1$, then insert $j+1$ into $\sigma_j$ within one of its cycles so that it goes in between two letters that create an ascent within $\text{Flatten}(\sigma_j)$ (where one selects the particular ascent in which to insert $j+1$ from left to right based off of the value of $b_j$ as before).  In either case, we let $\sigma_{j+1}$ denote the permutation of $[j+1]$ that results after $j+1$ has been inserted into $\sigma_j$.  After adding all of the letters from $[n]-\{1\}$ in this way, the permutation $\sigma_n$ results, which we define to be $g(\pi)$.  It may be verified that $g$ is the desired bijection between the sets $\mathcal{A}_n$ and $\mathcal{S}_n$.

For example, if $n=8$ and $$\pi=\{(0,2),(1,1),(0,3),(1,2),(0,3),(1,1),(0,5)\},$$
then we get
\begin{align*}
(1)&\rightarrow(1)(2)\rightarrow(13)(2)\rightarrow(13)(2)(4)\rightarrow(13)(25)(4)\rightarrow(13)(25)(46)\\
&\rightarrow(173)(25)(46)\rightarrow(173)(25)(46)(8),
\end{align*}
and thus $g(\pi)=(173)(25)(46)(8)$.

We next define a bijection $h$ between $\mathcal{A}_n$ and $\mathcal{S}_n$ in which the number of $1$'s occurring among the first coordinates of $\pi \in \mathcal{A}_n$, represented as above, corresponds to the number of big descents in $h(\pi)$. Let $\rho_1$ denote the permutation of length one.   We subsequently add members of $[n]-\{1\}$ using $\pi$ as an encoding to generate a sequence of permutations $\rho_2,\rho_3,\ldots,\rho_n$ (represented as words) as follows.  If $a_j=0$, where $j \in [n-1]$, then insert the letter $j+1$ into the permutation $\rho_j$ either between any two letters comprising a big descent or just before $j$ or at the very end (letting the value of $b_j$ dictate the action taken here, much as before).  If $a_j=1$, then we consider cases on whether or not the letter $j$ starts $\rho_j$.  If $j$ starts $\rho_j$, then insert $j+1$ in between any two letters of $\rho_j$ which do not form a big descent within $\rho_j$.  If $j$ does not start $\rho_j$, the
 n either add $j+1$ at the beginning or in between any two letters which do not form a big descent, except directly before $j$.  In either case, let the value of $b_j$ determine the position chosen for this insertion going from left to right as $b_j$ increases.  Let $h(\pi)=\rho_n$.  It may be verified that $h$ is the desired bijection.

For instance, if $\pi$ is as in the previous example, then we get
\begin{align*}
1\rightarrow 12 \rightarrow 312 \rightarrow 3124 \rightarrow 31524 \rightarrow 316524\rightarrow 7316524 \rightarrow 73165248,
\end{align*}
and thus $h(\pi)=73165248$.

The composition $h\circ g^{-1}$ then provides a bijection of $\mathcal{S}_n$ showing the equivalence of the flattened descent and big descent statistics.  Using the previous examples, if $\sigma=75162438=(173)(25)(46)(8)\in \mathcal{S}_8$, which has three flattened descents, then $h\circ g^{-1}(\sigma)=73165248$, which has three big descents. \end{proof}

\subsection{$d$-Descents}

Suppose $d \geq 1$ and $\sigma \in \mathcal{S}_n$, with $\text{Flatten}(\sigma)=\sigma_1\sigma_2\cdots\sigma_n$.  We will call an index $i$ such that $\sigma_i-\sigma_{i+1}\geq d$ a (flattened) $d$-descent.  For example, if $d=3$ and $\sigma=(1985)(24)(367) \in \mathcal{S}_9$, then there are two $3$-descents (at positions $3$ and $4$) and four descents altogether. Let $a_{n,m,k}=a_{n,m,k}(d)$ denote the number of permutations of $[n]$ having exactly $m$ $d$-descents and $k$ cycles.
Note that if $m>0$, then $a_{n,m,k}$ is non-zero only when $n \geq m+d+1$.

Let $c_{n,k}$ denote the signless Stirling number of the first kind which counts the permutations of $[n]$ having exactly $k$ cycles.  Note that if $1 \leq n \leq d+1$, then
$$a_{n,m,k}=\begin{cases}
c_{n,k},&\text{if }m=0;\\
0,&\text{if }m>0,
\end{cases}$$
as there is no restriction on the positions of the letters. If $n=d+2$, then $a_{d+2,1,k}=c_{d+1,k}$ and thus $a_{d+2,0,k}=c_{d+2,k}-c_{d+1,k}$, since the only possible $d$-descent in this case occurs with the letter $d+2$ coming just before $2$ once the permutation is flattened.  The following proposition provides a recurrence for $a_{n,m,k}$ when $n \geq d+3$.

\begin{proposition}\label{desp1}
If $n \geq d+3$ and $k \geq 1$, then
\begin{equation}\label{desp1e1}
a_{n,m,k}=a_{n-1,m,k-1}+(m+d)a_{n-1,m,k}+(n-m-d)a_{n-1,m-1,k}, \qquad m\geq0,
\end{equation}
where $a_{n,-1,k}=0$ and $a_{n,m,0}=0$ if $n>0$.
\end{proposition}
\begin{proof}
The first term on the right-hand side of \eqref{desp1e1} counts the permutations enumerated by $a_{n,m,k}$ in which $n$ occurs in a cycle by itself, while the second term counts those in which $n$ creates neither an additional $d$-descent nor cycle when it is added to a permutation of length $n-1$.  Note that the latter may be achieved only by adding $n$ between two letters comprising a $d$-descent or directly preceding some member of $[n-d+1,n-1]$ or at the end of the last cycle, whence there are $m+d$ choices altogether.  Finally, if $n$ is to create a new $d$-descent, then it must be inserted between two letters not comprising a $d$-descent, but not directly preceding a member of $[n-d+1,n-1]$.  For this, there are $(n-2)-(m-1)-(d-1)=n-m-d$ choices for the position of $n$, by subtraction, which completes the proof.
\end{proof}

Let $a_{n,m}=\sum_{k}a_{n,m,k}$ denote the number of permutations of size $n$ having exactly $m$ $d$-descents.  Summing \eqref{desp1e1} over $k$ implies
\begin{equation}\label{desp1e1a}
a_{n,m}=(m+d+1)a_{n-1,m}+(n-m-d)a_{n-1,m-1}, \qquad n>m \geq 0.
\end{equation}

Using \eqref{desp1e1a}, it is possible to show by induction the following explicit formula for $a_{n,m}$.

\begin{proposition}\label{desp2}
If $d \geq 1$, $m \geq 0$, and $n \geq m+d+1$, then
$$a_{n,m}=(d+1)!\sum_{\substack{i_1+i_2+\cdots+i_{m+1}=n-1-d-m\\
i_j\geq0}}(i_1+1)(i_1+i_2+1)\cdots(i_1+i_2+\cdots +i_m+1)\prod_{j=1}^{m+1}(d+j)^{i_j}.$$
\end{proposition}
\begin{proof}
We provide a combinatorial proof by showing that the right-hand side counts the permutations of length $n$ having exactly $m$ $d$-descents, the set of which we'll denote here by $S_{n,m}$.  To form a member of $S_{n,m}$, where $n \geq m+d+1$, we first write the members of $[d+1]$ as any permutation in standard cycle form, of which there are $(d+1)!$ possibilities.  Observe that there is no restriction on the positions of these letters since none of them may be the larger letter in a $d$-descent (note that $d+1$ cannot precede $1$ after a permutation has been flattened). Once the positions for the members of $[d+1]$ have been determined, we subsequently insert the members of $[d+2,n]$ to form a permutation of size $n$ in standard cycle form and classify each member according to whether or not it created an additional $d$-descent in the step at which it was inserted.  More specifically, we'll call $r \in [d+2,n]$ a \emph{producer} if $r$ produces an additional $d$-descent when
 inserted into the current permutation involving the letters in $[r-1]$  and a \emph{non-producer} if no additional $d$-descent is produced.  Note than in forming any member of $S_{n,m}$, there will be exactly $m$ producers among the elements of $[d+2,n]$ and thus $n-1-d-m$ non-producers.

Let $i_j$, $2 \leq j \leq m$, denote the number of non-producers between the $(j-1)$-st and $j$-th producers (with $i_1$ denoting the number of non-producers prior to the first producer and $i_{m+1}$ the number following the $m$-th producer).  Then the sets comprising the producers and the non-producers within a member of $S_{n,m}$ are uniquely determined by the vector $(i_1,i_2,\ldots,i_{m+1})$, where $i_1+i_2+\cdots+i_{m+1}=n-1-d-m$.  To complete the proof, it suffices to show that the number of permutations $\sigma \in S_{n,m}$ having fixed vector $(i_1,i_2,\ldots,i_{m+1})$ is given by $$(i_1+1)(i_1+i_2+1)\cdots(i_1+i_2+\cdots +i_m+1)\prod_{j=1}^{m+1}(d+j)^{i_j}$$ once the positions of the elements of $[d+1]$ have been specified.

Let $p_j$ denote the $j$-th producer within such a permutation $\sigma$.  Note first that
$$p_j=i_1+i_2+\cdots+i_j+j+d+1, \qquad 1 \leq j \leq m.$$
If $t$ is a non-producer coming between the $(j-1)$-st and $j$-th producers of $\sigma$, i.e., if $t \in [p_{j-1}+1,p_j-1]$, then $t$ must be inserted so that it either (i) goes between two consecutive letters comprising a current $d$-descent, (ii) precedes directly any letter in $[t-d+1,t-1]$, (iii) occurs at the end of the last cycle in the present permutation, or (iv) occurs as the $1$-cycle $(t)$.  In all, there are $(j-1)+(d-1)+2=j+d$ options regarding the placement of such $t$ and $p_j-p_{j-1}-1=i_j$ possible $t$, which implies that there are $(d+j)^{i_j}$ choices concerning the placement of all letters occurring between $p_{j-1}$ and $p_j$, $2 \leq j \leq m$.  Similarly, there are $(d+1)^{i_1}$ choices for the non-producers preceding $p_1$ and $(d+m+1)^{i_{m+1}}$ choices for those following $p_m$.  Finally, concerning the placement of the $j$-th producer $p_j$, note that it must be inserted between two members of $[p_{j}-1]$ but not between two letters that comprise a
 current $d$-descent or directly preceding some member of $[p_j-d+1,p_j-1]$.  Thus, there are $$(p_{j}-2)-(j-1)-(d-1)=i_1+i_2+\cdots+i_j+1$$
choices regarding the placement of the $j$-th producer for all $j$, which completes the proof.
\end{proof}

Noting that $n!=\sum_{m}a_{n,m}$ gives the following formula.

\begin{corollary}
If $d \geq 1$ and $n \geq d+1$, then
$$\frac{n!}{(d+1)!}=\sum_{m=0}^{n-d-1}\sum_{\substack{i_1+i_2+\cdots+i_{m+1}=n-1-d-m\\
i_j\geq0}}\prod_{j=1}^{m+1}(d+j)^{i_j}\prod_{j=1}^m\left(1+\sum_{\ell=1}^ji_\ell\right).$$
\end{corollary}

Finally, it is possible to extend Theorem \ref{thdes2} as follows using the combinatorial argument given above for it, the details of which we leave to the interested reader.

\begin{proposition}\label{desp3}
The statistic which records the number of flattened descents of size $d$ or more has the same distribution on $\mathcal{S}_n$ as does the one recording the number of typical descents of size $d+1$ or more for all $n,d \geq 1$.
\end{proposition}

\section{Counting $123$-subwords, $321$-subwords, peaks and valleys}\label{sec:3}

In this section, we count the occurrences of $123$-subwords, $321$-subwords,
peaks, and valleys within flattened permutations of length~$n$.
We will use the notation $\rr$, $\dd$, $\pp$ and $\vv$ to stand for these respective patterns.
In the unified approach described below,
we use the notation $\s$ to represent any one of these statistics.
Define
\[
g_n^\s(a_1a_2\cdots a_k)
=\sum_{\pi}q^{\s(\text{Flatten}(\pi))},
\]
where $\pi$ ranges over all the permutations of length~$n$ such that $\text{Flatten}(\pi)$ starts with the letters $a_1a_2\cdots a_k$. It is clear that $g_n^\s(a_1a_2\cdots a_k)= 0$ if $a_1\ne1$.
Write $g_n^\s=g_n^\s(1)$ for short.
For example, the notation $g_n^\rr$ represents the distribution
for the number of $123$-subwords in flattened permutations of length~$n$. By definition, we have
\begin{equation}\label{rec:g}
g_n^\s(a_1a_2\cdots a_k)
=\sum_{h\in[n]\backslash\{a_1,a_2,\ldots,a_k\}}g_n^\s(a_1a_2\cdots a_kh)
\end{equation}
for any $1 \leq k \leq n-1$.

Since the statistics $\rr,\dd,\pp,\vv$ involve patterns of length three,
we consider $k=4$ in~(\ref{rec:g}). This leads us to find all of the ~$g^\s_n(1ijk)$.
The following lemma allows us to concentrate on
those~$g^\s_n(a_1a_2\cdots a_k)$ with $k\le3$.

\begin{lemma}\label{lem:reduction}
Let $n\ge4$.
Suppose that $i,j,k$ are all different numbers in the set $\{2,3,\ldots,n\}$.
Then for any statistic $\s\in\{\rr,\dd,\pp,\vv\}$, we have
\[
g^\s_n(1ijk)
=\bigl(1+\chi(i=2)\bigr)q^{\s(1ijk)-\s(1jk)}\cdotp g^\s_{n-1}(1j'k'),
\]
where $j'=j-\chi(j>i)$ and $k'=k-\chi(k>i)$.
\end{lemma}

\begin{proof}
Fix $n\ge4$ and $i,j,k$.
Let $\pi$ be a permutation of length~$n$ such that~$\text{Flatten}(\pi)$ starts with~$1ijk$.
Let~$\psi(\pi)$ be the permutation of length~$n-1$
obtained from $\pi$ by removing the letter~$i$
in the cycle notation and replacing all of the remaining letters $h$ by $h'=h-\chi(h>i)$.
In other words, we delete the second letter in the cycle notation of $\pi$ and translate
the other letters preserving the order to form a permutation of length~$n-1$.
Since $i,j,k$ are fixed, it is easy to see that the map~$\psi$ is a bijection.
Since $\s$ is a statistic involving three letters, we may deduce
\begin{equation}\label{eq:reduction-1}
\s(\text{Flatten}(\pi))=\s(\text{Flatten}(\psi(\pi)))+\s(1ijk)-\s(1j'k')=\s(\text{Flatten}(\psi(\pi)))+\s(1ijk)-\s(1jk).
\end{equation}

Conversely, let~$\sigma$ be any permutation of length~$n-1$
such that $\text{Flatten}(\sigma)$ starts from $1j'k'$.
If $i\ge3$, then the letter~$i$ must be in the first cycle of $\pi$. In this case,
the inverse $\psi^{-1}(\sigma)=\pi$ starts with $1ijk$.
Otherwise, $i=2$. Then the letter $2$ may be in either the first or the second cycle of $\pi$.
So the inverse $\psi^{-1}(\sigma)$ consists of two permutations $\{\pi,\pi'\}$
such that $\text{Flatten}(\pi)=\text{Flatten}(\pi')$.
Therefore,
the permutation~$\sigma$ contributes twice in this case.

In summary, we have
\begin{align*}
g^\s_n(1ijk)
=\sum_{\pi}q^{\s(\text{Flatten}(\pi))}
&=\bigl(1+\chi(i=2)\bigr)\sum_{\sigma}q^{\s(\text{Flatten}(\sigma))+\s(1ijk)-\s(1jk)}\\
&=\bigl(1+\chi(i=2)\bigr)q^{\s(1ijk)-\s(1jk)}\cdotp g^\s_{n-1}(1j'k'),
\end{align*}
where $\pi$ ranges over all permutations of length~$n$
starting with $1ijk$ and $\sigma$ ranges over all permutations of length~$n-1$
starting with $1j'k'$.
This completes the proof.
\end{proof}

Let us now consider $g^\s_n(1ij)$.
Note that $g^\s_n(1ij)$ with $j>i$ can be reduced via
\begin{equation}\label{eq:cls-0}
g^\s_n(1ij)=g^\s_{n-1}\bigl(1(j-1)\bigr)\bigl(1-\theta\cdotp\chi(\s=\rr)\bigr),
\end{equation}
which has only one variable~$j$.
For the sake of reducing $g^\s_n(1ij)$ in the other case $j<i$,
we need an exchanging letters trick.

\begin{lemma}\label{lem:exchange}
Let $n\ge 4$.
For any statistic $\s\in\{\rr,\dd,\pp,\vv\}$ and any $2\le j<i\le n$,
we have
\[
g_n^\s(1ij)=g_n^\s\bigl(1(j+1)j\bigr).
\]
\end{lemma}

\begin{proof}
Let $2\le j<i$ and $\s\in\{\rr,\dd,\pp,\vv\}$.
Assume that $j\le i-2$. Let $\pi$ be a permutation of length~$n$.
Denote by $\pi'$ the permutation obtained from $\text{Flatten}(\pi)$ by exchanging the letters $i$ and $i-1$.
It is easy to see that $\s(\text{Flatten}(\pi))=\s(\text{Flatten}(\pi'))$.
Iterating in this way, we get
\[
g_n^\s(1ij)=g_n^\s\bigl(1(i-1)j\bigr)=\cdots=g_n^\s\bigl(1(j+1)j\bigr),
\]
which completes the proof.
\end{proof}

At this stage, we can construct a recurrence for the sequence $\{g^\s_n(1k)\}_{k=3}^n$
by finding $g^\s_n\bigl(1(i+1)i\bigr)$ in two different ways.
Let $3\le i\le n-1$ and $\s\in\{\rr,\dd,\pp,\vv\}$.
On one hand, by Lemma~\ref{lem:exchange} and~(\ref{eq:cls-0}), we have
\begin{align}
\sum_{2\le j\le i-1}g^\s_n\bigl(1(j+1)j\bigr)
&=\sum_{2\le j\le i-1}g^\s_n(1ij)
=g^\s_n(1i)-\sum_{j>i}g^\s_n(1ij)\notag\\
&=g^\s_n(1i)-\sum_{j\ge i}g^\s_{n-1}(1j)\bigl(1-\theta\cdotp\chi(\s=\rr)\bigr).\label{eq:cls-1}
\end{align}
Then the first-order difference transformation of the above formula gives us
\begin{equation}\label{fm:g1i+1i-1}
g^\s_n\bigl(1(i+1)i\bigr)
=g^\s_n\bigl(1(i+1)\bigr)-g^\s_n(1i)+g^\s_{n-1}(1i)\bigl(1-\theta\cdotp\chi(\s=\rr)\bigr).
\end{equation}
On the other hand, applying Lemma~\ref{lem:reduction} to $g^\s_n\bigl(1(i+1)ik\bigr)$,
we find
\[
g^\s_n\bigl(1(i+1)ik\bigr)
=\begin{cases}
\bigl[1-\theta\cdotp\chi(\s=\dd)\bigr]\cdotp g^\s_{n-1}\bigl(1(k+1)k\bigr),
&\text{if }2\le k\le i-1;\\[3pt]
\bigl[1-\theta\cdotp\chi(\s\in\{\pp,\vv\})\bigr]\cdotp g^\s_{n-2}\bigl(1(k-2)\bigr),
&\text{if }k\ge i+2.
\end{cases}
\]
Therefore, by using~(\ref{eq:cls-1}), we may deduce
\begin{align}
g^\s_n\bigl(1(i+1)i\bigr)
&=\sum_{2\le k\le i-1}g^\s_n\bigl(1(i+1)ik\bigr)
+\sum_{k\ge i+2}g^\s_n\bigl(1(i+1)ik\bigr)\notag\\
&=\bigl[1-\theta\cdotp\chi(\s=\dd)\bigr]\cdotp g^\s_{n-1}(1i)
+\bigl[1-2\chi(\s\in\{\pp,\vv\})\bigr]\cdotp\theta\cdotp\sum_{k\ge i}g^\s_{n-2}(1k).\label{fm:g1i+1i-2}
\end{align}
Equating~(\ref{fm:g1i+1i-1}) and~(\ref{fm:g1i+1i-2}),
we obtain a recurrence
\begin{align*}
&\theta^{-1}\cdotp\Bigl[g^\s_n\bigl(1(i+1)\bigr)-g^\s_n(1i)\Bigr]\\
=\ &\bigl[\chi(\s=\rr)-\chi(\s=\dd)\bigr]\cdotp g^\s_{n-1}(1i)
+\bigl[1-2\chi(\s\in\{\pp,\vv\})\bigr]\cdotp\sum_{k\ge i}g^\s_{n-2}(1k).
\end{align*}

Let $3\le i\le n-2$. For the sake of eliminating the sum,
we apply the first-order difference transformation to it and get
\begin{align}
&\theta^{-1}\cdotp\Bigl[g^\s_n\bigl(1(i+2)\bigr)-2g^\s_n\bigl(1(i+1)\bigr)+g^\s_n(1i)\Bigr]\notag\\
=\ &\bigl[\chi(\s=\rr)-\chi(\s=\dd)\bigr]\cdotp \Bigl[g^\s_{n-1}\bigl(1(i+1)\bigr)-g^\s_{n-1}(1i)\Bigr]
-\bigl[1-2\chi(\s\in\{\pp,\vv\})\bigr]\cdotp g^\s_{n-2}(1i).\label{rec}
\end{align}
%Or equivalently, for any $5\le k\le n$, we have
%\begin{multline}\label{rec}
%\theta^{-1}\cdotp\Bigl[g^\s_n(1k)-2g^\s_n\bigl(1(k-1)\bigr)+g^\s_n\bigl(1(k-2)\bigr)\Bigr]\\
%=\bigl[\chi(\s=\rr)-\chi(\s=\dd)\bigr]\cdotp \Bigl[g^\s_{n-1}\bigl(1(k-1)\bigr)-g^\s_{n-1}\bigl(1(k-2)\bigr)\Bigr]\\
%-\bigl[1-2\chi(\s\in\{\pp,\vv\})\bigr]\cdotp g^\s_{n-2}\bigl(1(k-2)\bigr).
%\end{multline}
Our arguments will be based on the above recurrence.

\subsection{$123$-Subwords}\label{ssec:123}

In this case, the recurrence (\ref{rec}) reads
\begin{equation}\label{rec:123-g1k}
g^\rr_n(1k)
-2g^\rr_n\bigl(1(k-1)\bigr)
+g^\rr_n\bigl(1(k-2)\bigr)
-\theta \Bigl[
g^\rr_{n-1}\bigl(1(k-1)\bigr)
- g^\rr_{n-1}\bigl(1(k-2)\bigr)
-g^\rr_{n-2}\bigl(1(k-2)\bigr)
\Bigr]=0,
\end{equation}
for any $5\le k\le n$.
Recall that $U_n(x)$ is the $n$-th Chebyshev polynomial of the second kind.
We will now solve the recurrence (\ref{rec:123-g1k}) for $3\le k\le n-1$, with the initiation
\begin{equation}\label{ini:123-g1i}
\begin{split}
g^\rr_n(13)&=(1-\theta)g^\rr_{n-1}+2\theta(2-\theta)g^\rr_{n-2},\qquad n\ge4,\\
g^\rr_n(14)&=(1-\theta)g^\rr_{n-1}+3\theta(2-\theta)g^\rr_{n-2}-2\theta(1-3\theta+\theta^2)g^\rr_{n-3},\qquad n\ge5.
\end{split}
\end{equation}

\begin{theorem}\label{thm:123-g1k}
For any $3\le k\le n-1$, we have
\begin{equation}\label{fm:123-g1k}
g^\rr_n(1k)=\sum_{i=1}^{k-1}\alpha_{k,\,i}\cdot b^\rr_i\cdot g^\rr_{n-i},
\end{equation}
where
$\alpha_{k,\,i}={k-1\choose i-1}-{k-3\choose i-3}$ and
$b^\rr_i=-\theta^{(i-1)/2}U_{i+1}\bigl(\sqrt\theta /2\bigr)$.
\end{theorem}

\begin{proof}
It is straightforward to check that formula~(\ref{fm:123-g1k})
admits the initiation~(\ref{ini:123-g1i}).
It suffices to show that it satisfies the recurrence~(\ref{rec:123-g1k}) as well.
Let $3\le k\le n-1$.
Note that $\alpha_{k,i}=0$ if $i\le 0$ or $i\ge k$.
For notational convenience, we rewrite~(\ref{fm:123-g1k}) as
$g^\rr_n(1k)=\sum_{i}\alpha_{k,\,i}\cdot b^\rr_i\cdot g^\rr_{n-i}$,
where $i$ runs over all integers.
Denote the left-hand side of~(\ref{rec:123-g1k}) by $L_{n,k}$.
Since $\alpha_{k+1,\,i}-\alpha_{k,\,i}=\alpha_{k,\,i-1}$
for any~$i$ and
\begin{equation}
b^\rr_i-\theta(b^\rr_{i-1}-b^\rr_{i-2})
=-\theta^{(i-1)/2}\Bigl[U_{i+1}(\sqrt\theta /2)-\sqrt\theta \cdot U_{i}(\sqrt\theta /2)+U_{i+1}(\sqrt\theta /2)\Bigr]=0,
\end{equation}
we may deduce
\begin{align*}
L_{n,k}&=\sum_{i=1}^{k+1}\alpha_{k+2,i}\,b^\rr_i\,g^\rr_{n-i}
-2\sum_{i=1}^k\alpha_{k+1,i}\,b^\rr_i\,g^\rr_{n-i}
+\sum_{i=1}^{k-1}\alpha_{k,i}\,b^\rr_i\,g^\rr_{n-i}\\
&\quad -\theta\Biggl[
\sum_{i=1}^{k}\alpha_{k+1,i}\,b^\rr_i\,g^\rr_{n-1-i}
-\sum_{i=1}^{k-1}\alpha_{k,i}\,b^\rr_i\,g^\rr_{n-1-i}
-\sum_{i=1}^{k-1}\alpha_{k,i}\,b^\rr_i\,g^\rr_{n-2-i}\Biggr]\\
&=\sum_{i\ge1}\alpha_{k,i-2}\cdotp\Bigl[b^\rr_i-\theta(b^\rr_{i-1}-b^\rr_{i-2})\Bigr]\cdotp g^\rr_{n-i}
=0.
\end{align*}
This completes the proof.
\end{proof}

\begin{corollary}
For any $n\ge4$, we have
\begin{equation}\label{fm:123-gn}
g^\rr_n=\sum_{i=1}^{n-2}g^\rr_{n-i}\Bigl(
\alpha_{n,i+1}b^\rr_i-\theta\,\alpha_{n-1,\,i} b^\rr_{i-2}
\Bigr).
\end{equation}
\end{corollary}

\begin{proof} It is routine to verify \eqref{fm:123-gn} for $n=4$. Let $n\ge5$.
By~(\ref{rec:123-g1k}) and~(\ref{fm:123-g1k}), we can derive a formula for $g^\rr_n(1n)$ as follows:
\begin{align*}
g^\rr_n(1n)&=g^\rr_n-g^\rr_n(12)-\sum_{3\le k\le n-1}g^\rr_n(1k)\\
&=g^\rr_n-2q\cdotp g^\rr_{n-1}-
\sum_{3\le k\le n-1}\biggl(q\cdotp g^\rr_{n-1}+
\sum_{i\ge2}\alpha_{k,i}\,b^\rr_i\,g^\rr_{n-i}\biggr)\\
&=g^\rr_n-(n-1)q\cdotp g^\rr_{n-1}-
\sum_{2\le i\le n-2}\alpha_{n,i+1}b^\rr_i\,g^\rr_{n-i}\\
&=-\sum_{i\ge0}\alpha_{n,i+1}b^\rr_i\,g^\rr_{n-i}.
\end{align*}
Using the above formula,
and taking $k=n$ in recurrence~(\ref{rec:123-g1k}), we find
\begin{align*}
0=&g^\rr_n(1n)
-2g^\rr_n\bigl(1(n-1)\bigr)
+g^\rr_n\bigl(1(n-2)\bigr)\\[3pt]
&\quad-\theta \Bigl[
g^\rr_{n-1}\bigl(1(n-1)\bigr)
- g^\rr_{n-1}\bigl(1(n-2)\bigr)
-g^\rr_{n-2}\bigl(1(n-2)\bigr)
\Bigr]\\[3pt]
=&-\sum_{i\ge0}\alpha_{n,i+1}\,b^\rr_i\,g^\rr_{n-i}
-2\sum_{i\ge0}\alpha_{n-1,\,i}\,b^\rr_i\,g^\rr_{n-i}
+\sum_{i\ge0}\alpha_{n-2,\,i}\,b^\rr_i\,g^\rr_{n-i}\\
&-\theta\Biggl[
-\sum_{i\ge0}\alpha_{n-1,i+1}\,b^\rr_i\,g^\rr_{n-1-i}
-\sum_{i\ge0}\alpha_{n-2,\,i}\,b^\rr_i\,g^\rr_{n-1-i}
+\sum_{i\ge0}\alpha_{n-2,i+1}\,b^\rr_i\,g^\rr_{n-2-i}\Biggr]\\
=&\sum_{i\ge0}g^\rr_{n-i}\Bigl[
(-\alpha_{n,i+1}-2\alpha_{n-1,\,i}+\alpha_{n-2,\,i})b^\rr_i
-\theta\bigl(
(-\alpha_{n-1,i}-\alpha_{n-2,\,i-1})\,b^\rr_{i-1}
+\alpha_{n-2,i-1}\,b^\rr_{i-2}
\bigr)\Bigr]\\
=&\sum_{i\ge0}g^\rr_{n-i}\Bigl(
\theta\,\alpha_{n-1,\,i} b^\rr_{i-2}-\alpha_{n,i+1}b^\rr_i
\Bigr),
\end{align*}
from which~(\ref{fm:123-gn}) follows immediately.
\end{proof}

To proceed further, we will need the generating function of $b^\rr_n$.
By~(\ref{fm:U}), we have
\[
\sum_{n\geq1}U_{n+1}(t)\frac{x^n}{n!}
={A(t)^2e^{A(t)x}-B(t)^2e^{B(t)x}\over A(t)-B(t)}-2t,
\]
where
$A(t)=t+\sqrt{t^2-1}$ and $B(t)=t-\sqrt{t^2-1}$.
This gives us the generating function
\begin{align*}
B^\rr(x)&=\sum_{n\geq1}b^\rr_n\frac{x^n}{n!}
=-\frac{1}{\sqrt\theta }\sum_{n\geq1}
U_{n+1}\left(\frac{\sqrt\theta }{2}\right)\frac{(\sqrt\theta x)^n}{n!}\\
&=1+\frac{(\sqrt\theta -\sqrt{\theta-4})^2e^{(\sqrt\theta -\sqrt{\theta-4})\sqrt\theta x/2}
-(\sqrt\theta +\sqrt{\theta-4})^2e^{(\sqrt\theta +\sqrt{\theta-4})\sqrt\theta x/2}}{4\sqrt\theta \sqrt{\theta-4}}.
\end{align*}
By using Euler's formula $e^{i\phi}=\cos\phi+i\sin\phi$,
one may show that
\begin{equation}\label{fm:123-B}
B^\rr(x)=1-\frac {2{{\rm e}^{\theta x/2}}}{\sqrt{\theta(4-\theta)}}
\cos\Biggl( \frac{\sqrt{\theta(4-\theta)}x}{2}+\arcsin\biggl( \frac{2-\theta}{2}\biggr) \Biggr).
\end{equation}

We now give an explicit formula for $G^\rr(x)$.

\begin{theorem}\label{thm:123-main}
We have
$G^\rr(x)=1+\int_{0}^xH(t)\,dt$, where
\begin{equation}\label{fm:123-H}
H(x)={2\sqrt\theta(4-\theta)(1+\xi^2)^2
\Bigl[(1-\xi^2)\sqrt{4-\theta}\bigl(1-(1-\theta)e^{-\theta x}\bigr)
+2\sqrt\theta\xi\bigl(1-(3-\theta)e^{-\theta x}\bigr)\Bigr]
\over (\sqrt\theta\xi+\sqrt{4-\theta})^3(\sqrt\theta-\sqrt{4-\theta}\xi)^3},
\end{equation}
with $\xi=\tan(\sqrt{\theta(4-\theta)}x/4)$.
\end{theorem}

\begin{proof}
Recall that $G^\rr(x)=\sum_{n\ge1}g^\rr_{n}{x^{n-1}\over (n-1)!}$.
Since
\[
\alpha_{n,i}={n-1\choose i-1}-{n-3\choose i-3}={(n-3)!\bigl[(n-1)(n-2)-(i-1)(i-2)\bigr]\over (i-1)!(n-i)!},
\]
multiplying~(\ref{fm:123-gn})
by ${x^{n-4}\over (n-4)!}$ and summing over $n\ge4$ gives us
\begin{align*}
\sum_{n\ge4}g^\rr_n{x^{n-4}\over (n-4)!}
&=\sum_{n\ge4}\sum_{i=1}^{n-2}g^\rr_{n-i}\alpha_{n,i+1}b^\rr_i{x^{n-4}\over (n-4)!}
-\theta\sum_{n\ge4}\sum_{i=1}^{n-2}g^\rr_{n-i}\alpha_{n-1,\,i} b^\rr_{i-2}{x^{n-4}\over (n-4)!}\\
&=\sum_{n\ge4}{n-3\over x^3}
\sum_{i=1}^{n-2}
{b^\rr_ix^i\over i!}\cdot
{g^\rr_{n-i}x^{n-i-1}\over (n-i-1)!}\cdot
\bigl[(n-1)(n-2)-i(i-1)\bigr]\\
&\quad-\theta\sum_{n\ge4}{1\over x^3}
\sum_{i=1}^{n-2}
{b^\rr_{i-2}x^i\over (i-1)!}\cdot
{g^\rr_{n-i}x^{n-i-1}\over (n-i-1)!}\cdot
\bigl[(n-2)(n-3)-(i-1)(i-2)\bigr].
\end{align*}
Upon noting $b^\rr_{-1}=-1/\theta$ and $b^\rr_0=-1$, the preceding equation may then be expressed as
\begin{align*}
{d^3\over dx^3}G^\rr(x)
&={d^3\over dx^3}\Bigl(\bigl(G^\rr(x)-1\bigr)B(x)\Bigr)
-{d\over dx}\Bigl(\bigl(G^\rr(x)-1\bigr)\cdot{d^2\over dx^2}B(x)\Bigr)\\
&\quad
-\theta{d^2\over dx^2}
\biggl(\bigl(G^\rr(x)-1\bigr)\Bigl(-{1\over\theta}-x+\int_{0}^xB(t)\,dt\Bigr)\biggr)
+\theta\Bigl(\bigl(G^\rr(x)-1\bigr)\cdot{d\over dx}B(x)\Bigr).
\end{align*}
Letting $H(x)=\frac{d}{dx}G^\rr(x)$ reduces the last equation to
{\small\begin{align*}
2\biggl(\theta-\theta B(x)+{d^2\over dx^2}B(x)\biggr)H(x)
+\biggl(1+\theta x+{3d\over dx}B(x)-\theta\int_{0}^xB(t)\,dt\biggr){d\over dx}H(x)
+\bigl(B(x)-1\bigr){d^2\over dx^2}H(x)
=0.
\end{align*}}With the boundary values $H(0)=2$ and $H'(0)=6-4\theta$,
we obtain the solution~(\ref{fm:123-H}) from the preceding differential equation. Consequently,
we obtain $G^\rr(x)$, upon noting $G^\rr(0)=1$.
\end{proof}

After integration of $H$ and several algebraic operations, one may derive the further formula
\begin{align*}
G^\rr(x)=\frac {2\sqrt{\theta}\cos\left(7\gamma +2\beta x \right) -4\cos \left(2\gamma +\beta x \right)-4+\theta}{4 \left( 1-\cos \left( 4\gamma +\beta x \right)  \right) ^{2}}
 +\frac {6-(4-\theta) e^{-\theta x} }{2 \left( 1-\cos \left( 4\gamma +\beta x \right)  \right)},
\end{align*}
where $\beta=\sqrt {\theta(4-\theta)}$ and $\gamma=\arccos \frac{\sqrt{\theta}}{2}$.
From this, one can find an explicit formula for $g^\rr_n$ using the Taylor expansion of $G^\rr(x)$, an exercise we leave to the interested reader.

\indent Differentiation of $G^r(x)$ with respect to $q$ yields the following result.

\begin{corollary}\label{cor123}
If $n \geq 3$, then the average number of occurrences of $123$-subwords in $\text{Flatten}(\pi)$ over $\pi\in S_n$ is given by $\frac{n^2+3n-6}{6n}$.
\end{corollary}

\subsection{$321$-Subwords}\label{ssec:321}

Let $n\ge5$.
Then the recurrence (\ref{rec}) reads
for any $3\le k\le n-2$ as
\[
g^\dd_n\bigl(1(k+2)\bigr)
-2g^\dd_n\bigl(1(k+1)\bigr)
+g^\dd_n(1k)
+\theta \Bigl[
g^\dd_{n-1}\bigl(1(k+1)\bigr)
- g^\dd_{n-1}(1k)
+g^\dd_{n-2}(1k)
\Bigr]=0.
\]
It is easy to verify
$g^\dd_n(13)=g^\dd_{n-1}$ and
$g^\dd_n(14)=g^\dd_{n-1}-2\theta\cdotp g^\dd_{n-3}$.
Similar to the proof of Theorem~\ref{thm:123-g1k} above, one can derive a formula for $g_n^\dd$ as follows.
\begin{theorem}\label{thm:321-g1k}
For any $n\ge k\ge 3$, we have
\[
g^\dd_n(1k)=\sum_{i=1}^{k-1}\alpha_{k,\,i}\cdot b^\dd_i\cdot g^\dd_{n-i},
\]
where
$b^\dd_i=(-1)^{i}\theta^{(i-1)/2}U_{i-3}\bigl(\sqrt\theta /2\bigr)$.
Moreover,
$G^\dd(x)=\bigl(1-B^\dd(x)\bigr)^{-2}$,
where
\begin{align*}
B^\dd(x)&=1+\frac{(\sqrt\theta -\sqrt{\theta-4})^2e^{-(\sqrt\theta +\sqrt{\theta-4})\sqrt\theta  x/2}
-(\sqrt\theta +\sqrt{\theta-4})^2e^{-(\sqrt\theta -\sqrt{\theta-4})\sqrt\theta  x/2}}
{4\sqrt{\theta(\theta-4)}}.
\end{align*}
Consequently, for any $n\geq1$, we have
\[
g^\dd_n
=(\theta-4)\sqrt\theta ^{n+1}
\sum_{j\geq1}\ j\!\cdot\!
\left({\sqrt\theta +\sqrt{\theta-4}\over 2}\right)^{4j+4}\!\!\cdot\!
\Bigl(\sqrt\theta +j\sqrt{\theta-4}\Bigr)^{n-1}.
\]
\end{theorem}
Similar to~(\ref{fm:123-B}), we have
\[
B^\dd(x)=1-\frac {2{{\rm e}^{-\theta x/2}}}{\sqrt{\theta(4-\theta)}}
\cos\Biggl( \frac{\sqrt{\theta(4-\theta)}x}{2}+\arcsin\biggl( \frac{2-\theta}{2}\biggr) \Biggr).
\]
From this, one may obtain another expression for $g^\dd_n$ as
\[
g^\dd_n=(-\theta)^{n+1}\sum_{j\geq1}\sum_{\ell=1}^{2j-1}\frac{B_{2j}(-4)^j(1-4^j)\binom{n-1}{\ell-1}}{2j(2j-1-\ell)!}\left( \frac{\sqrt{4-\theta}}{-2\sqrt{\theta}}\right)^{\ell+1}\arcsin^{2j-1-\ell}\left( \frac{2-\theta}{2}\right),
\]
where $B_n$ is the $n$-th Bernoulli number.

\begin{corollary}\label{cor321}
If $n \geq 2$, then the average number of occurrences of $321$-subwords in $\text{Flatten}(\pi)$ over $\pi\in S_n$ is given by $\frac{(n-2)(n-3)}{6n}$.
\end{corollary}
\begin{proof}
Differentiating $G^\dd(x)$ with respect to $q$ and substituting $q=1$ yields
$$\frac{d}{dq}G^\dd(x)\mid_{q=1}=\frac{x^3}{3(1-x)^3}.$$
Thus, the total number of occurrences of $321$-subwords in  $\text{Flatten}(\pi)$ over all $\pi\in S_n$ is given by $(n-2)!\binom{n-1}{3}$ for $n\geq2$, which completes the proof.
\end{proof}

\subsection{Peaks}\label{ssec:peak}

Let $n\ge5$.
Then the  recurrence (\ref{rec}) reads
for any $3\le k\le n-2$ as
\[
g^\pp_n\bigl(1(k+2)\bigr)
-2g^\pp_n\bigl(1(k+1)\bigr)
+g^\pp_n(1k)
-\theta\cdotp g^\pp_{n-2}(1k)=0.
\]
With the initiation
$g^\pp_n(13)=g^\pp_{n-1}-2\theta\cdotp g^\pp_{n-2}$ and
$g^\pp_n(14)=g^\pp_{n-1}-3\theta\cdotp g^\pp_{n-2}+2\theta\cdotp g^\pp_{n-3}$,
we can derive the next result.
\begin{theorem}\label{thm:peak-gf}
For any $n\ge k\ge 3$, we have
\[
g^\pp_n(1k)=\sum_{i=1}^{k-1}\beta_{k,\,i}\cdot b^\pp_i\cdot g^\pp_{n-i},
\]
where
$\beta_{k,\,i}={k-2\choose i-1}+{k-3\choose i-2}$
and $b^\pp_i=(-1)^{i+1}\theta^{\lfloor i/2\rfloor}$.
Moreover,
$G^\pp(x)=\bigl(1-B^\pp(x)\bigr)^{-2}$,
where
\[
B^\pp(x)=\sum_{n\geq1}b^\pp_n\frac{x^n}{n!}
=1+{\sinh(\sqrt\theta  x)\over \sqrt\theta}-\cosh(\sqrt\theta  x).
\]
\end{theorem}

We shall give two formulas for~$g_n^\pp$. For the first one, we will need the following proposition whose proof is straightforward.

\begin{proposition}\label{prop:PB}
Let $G(x)=\sum_{n\ge0}g_{n+1}{x^n\over n!}$ and $B(x)=\sum_{n\ge1}b_n{x^n\over n!}$
be two generating functions. If
$G(x)=\bigl(1-B(x)\bigr)^{-2}$,
then for any $n\ge1$, we have
\[
g_{n+1}
=\sum_{k=1}^n(k+1)\sum_{i_1+i_2+\cdots+i_k=n\atop{i_1,\,i_2,\,\ldots,\,i_k\ge1}}
{n\choose i_1,\,i_2,\,\ldots,\,i_k}
b_{i_1}b_{i_2}\cdots b_{i_k}.
\]
\end{proposition}

For the second formula, we need a sequence $t_{n,m}$ defined by
\begin{align*}
t_{n,m}&=[x^n](x^2\cot(x))^m\\
&=2^{n-2m}(-1)^{(n-m)/2}m!\sum_{\ell=0}^m\sum_{k=0}^{n-2m+1}\frac{2^\ell k!s_1(\ell+k,\ell)s_2(n-2m+\ell,k)}{(m-\ell)!(\ell+k)!(n-2m+\ell)!},
\end{align*}
where $s_1(n,k)$ and $s_2(n,k)$ are the Stirling numbers of the first and second kind, respectively; see Sequence A199542 in \cite{OEIS}.
It follows that
\begin{equation}\label{eq:t}
[x^n](x^2\coth(x))^m=(-1)^{(n-m)/2}t_{n,m}.
\end{equation}

Now we can give the two formulas for~$g_n^\pp$.
\begin{theorem}
For any $n\geq1$, we have
\begin{align}
g_n^\pp
&=\sum_{k=1}^{n-1}(k+1)\sum_{i_1+i_2+\cdots+i_k=n-1\atop{i_1,\,i_2,\,\ldots,\,i_k\ge1}}
(-1)^{n+k-1}{n-1\choose i_1,\,i_2,\,\ldots,\,i_k}
\theta^{\lfloor i_1/2\rfloor+\lfloor i_2/2\rfloor+\cdots+\lfloor i_k/2\rfloor}\label{fm:peak-1}\\
&=n!\sum_{j\geq1}(-1)^{(n+j+2)/2}\cdot \sqrt\theta ^{n+j}t_{n+2j,\,j}.\label{fm:peak-2}
\end{align}
\end{theorem}

\begin{proof}
The formula~(\ref{fm:peak-1}) follows immediately from Theorem~\ref{thm:peak-gf}
and Proposition~\ref{prop:PB}.
To show~(\ref{fm:peak-2}), we deduce from~(\ref{eq:t}) that
\begin{align*}
G^\pp(x)&=\frac{1}{(1-B(x))^2}=\frac{1}{(\cosh(\sqrt\theta x)-\frac{\sinh(\sqrt\theta x)}{\sqrt\theta })^2}\\
&=\frac{d}{dx}\frac{1}{\sqrt\theta \coth(\sqrt\theta x)-1}
=-\frac{d}{dx}\sum_{j\geq1}\sqrt\theta ^{j}\coth^{j}(\sqrt\theta x)\\
&=-\frac{d}{dx}\sum_{j\geq1}\sqrt\theta ^{\,-j}x^{-2j}\bigl(\theta x^2\coth(\sqrt\theta x)\bigr)^{j}\\
&=\sum_{j\geq1}\sum_{i\geq j}(i-2j)(-1)^{(i-j)/2+1}\sqrt\theta ^{i-j}t_{i,\,j}\cdotp x^{i-2j-1}.
\end{align*}
Extracting the coefficient of $x^{n-1}$ gives~(\ref{fm:peak-2}) and completes the proof.
\end{proof}

\begin{corollary}\label{corpeak}
If $n \geq 2$, then the average number of peaks in $\text{Flatten}(\pi)$ over $\pi\in S_n$ is given by $\frac{n-2}{3}$.
\end{corollary}
\begin{proof}
Differentiating $G^\pp(x)$ with respect to $q$ and substituting $q=1$ yields
$$\frac{d}{dq}G^\pp(x)\mid_{q=1}=\frac{x^2(3-x)}{3(1-x)^3}.$$
Thus, the total number of peaks in  $\text{Flatten}(\pi)$ over all $\pi\in S_n$ is given by $\frac{n-2}{3}n!$ for $n\geq3$, which completes the proof.
\end{proof}

\subsection{Valleys}\label{ssec:valley}

Let $n\ge5$.
Then the recurrence (\ref{rec}) has the same form for $\vv$ as it does for $\pp$.
So for any $3\le k\le n-2$, we have
\[
g^\vv_n\bigl(1(k+2)\bigr)
-2g^\vv_n\bigl(1(k+1)\bigr)
+g^\vv_n(1k)
-\theta\cdotp g^\vv_{n-2}(1k)=0,
\]
with $g^\vv_n(13)=g^\vv_{n-1}-2\theta\cdotp g^\vv_{n-2}$ and
$g^\vv_n(14)=g^\vv_{n-1}-3\theta\cdotp g^\vv_{n-2}+2\theta\cdotp g^\vv_{n-3}$.
Note that the formula for $g^\vv_n(13)$ does not hold when $n=3$,
in contrast to the situation for peaks, which causes $g^\vv_n(1k)$ to differ from $g^\pp_n(1k)$.

\begin{theorem}\label{thm:valley-gf}
For any $n\ge k\ge3$, we have
\begin{equation}\label{fm:valley-g1k}
g^\vv_n(1k)=\sum_{i=1}^{k-1}\beta_{k,\,i}\cdot b^\vv_i\cdot g^\vv_{n-i}
+2\sqrt\theta^{n-1}\cdot\chi(\text{$n$ is odd and $k=n$}),
\end{equation}
where
$b^\vv_i=(-1)^{i+1}\theta^{\lfloor i/2\rfloor}$.
\end{theorem}

Since $b^\vv_n=b^\pp_n$,
the generating function $B^\vv(x)=\sum_{n\geq1}b^\vv_n\frac{x^n}{n!}$
is given by $B^\vv(x)=B^\pp(x)$.

\begin{theorem}\label{thm:valley-main}
We have
\[
G^\vv(x)=\frac{\sqrt\theta \bigl(-\sinh(2\sqrt\theta x)+\sqrt\theta \cosh(2\sqrt\theta x)+2\sqrt\theta x+\sqrt\theta \bigr)}{2\bigl(-\sinh(\sqrt\theta x)+\sqrt\theta \cosh(\sqrt\theta x)\bigr)^2}.
\]
Moreover, for any $n\ge2$, we have
\begin{equation}\label{fm:valley-gn}
g^\vv_n
=n(2\sqrt\theta)^n
\sum_{j\ge1}j^{n-1}{(1-\sqrt\theta)^{j-1}\over(1+\sqrt\theta)^{j+1}}.
\end{equation}
\end{theorem}

\begin{proof}
By~(\ref{fm:valley-g1k}), we deduce
\begin{align}\label{rec:valley-gn}
g^\vv_n=\sum_{i=1}^{n-1}\beta_{n+1,\,i+1}\cdot b^\vv_i\cdot g^\vv_{n-i}
+2\sqrt\theta^{n-1}\chi(\text{$n$ is odd}).
\end{align}
Multiplying~(\ref{rec:valley-gn}) by ${x^{n-2}\over (n-2)!}$, and summing over $n\geq2$, yields
\[
\frac{d}{dx}G^\vv(x)
=\sum_{n\geq2}\sum_{i=1}^{n-1}i\cdot\frac{b^\vv_i\cdot g^\vv_{n-i}\cdot x^{n-2}}{i!\cdot(n-1-i)!}+\sum_{n\geq2}\sum_{i=1}^{n-1}(n-1)\cdot\frac{b^\vv_i\cdot g^\vv_{n-i}\cdot x^{n-2}}{i!\cdot(n-1-i)!}
+2\!\!\sum_{n\ge1\atop{\text{$n$ is odd}}}\sqrt\theta^{n+1}{x^n\over n!},
\]
which is equivalent to
\[
\frac{d}{dx}G^\vv(x)=G^\vv(x)\frac{d}{dx}B^\vv(x)+\frac{d}{dx}\bigl(G^\vv(x)B^\vv(x)\bigr)+2\sqrt\theta \sinh(\sqrt\theta x).
\]
Solving this equation gives
\begin{align*}
G^\vv(x)&=\frac{\sqrt\theta\bigl(-\sinh(2\sqrt\theta x)+\sqrt\theta \cosh(2\sqrt\theta x)+2\sqrt\theta x+\sqrt\theta\bigr)}{2\bigl(\sqrt\theta \cosh(\sqrt\theta x)-\sinh(\sqrt\theta x)\bigr)^2}\\
&={d\over dx}\biggl({\sqrt{\theta}x(e^{2\sqrt{\theta}x}+1)\over (\sqrt{\theta}-1)e^{2\sqrt{\theta}x}+\sqrt{\theta}+1}\biggr).
\end{align*}
With the aid of~(\ref{gf:E}), we obtain
\begin{align*}
G^\vv(x)&=\frac{d}{dx}\left(x+\frac{x}{1-\sqrt\theta }
E\biggl(x(1+\sqrt\theta ),\,{1-\sqrt\theta\over 1+\sqrt\theta}\biggr)\right)\\
&=1+\sum_{n\geq1}\sum_{k\ge1}(n+1)E_{n,k}(1-\sqrt\theta )^{k-1}(1+\sqrt\theta )^{n-k}\frac{x^n}{n!}.
\end{align*}
Extracting the coefficient of $x^{n-1}$ yields
\begin{align*}
g^\vv_n
&=n(1+\sqrt\theta )^{n-2}\sum_{k\ge1}E_{n-1,k}\left(\frac{1-\sqrt\theta }{1+\sqrt\theta }\right)^{k-1}\\
&=n(1+\sqrt\theta )^{n-2}A_{n-1}\biggl({1-\sqrt\theta\over 1+\sqrt\theta}\biggr)\\
&=n(2\sqrt\theta)^n
\sum_{j\ge1}j^{n-1}{(1-\sqrt\theta)^{j-1}\over(1+\sqrt\theta)^{j+1}},
\end{align*}
which completes the proof.
\end{proof}

\begin{corollary}\label{corva}
If $n \geq 3$, then the average number of valleys in $\text{Flatten}(\pi)$ over $\pi\in S_n$ is given by $\frac{n-3}{3}$.
\end{corollary}
\begin{proof}
Differentiating $G^\vv(x)$ with respect to $q$ and substituting $q=1$ yields
$$\frac{d}{dq}G^\vv(x)\mid_{q=1}=\frac{2x^3(2-x)}{3(1-x)^3}.$$
Thus, the total number of valleys in  $\text{Flatten}(\pi)$ over all $\pi\in S_n$ is given by $\frac{n-3}{3}n!$ for $n\geq3$, which completes the proof.
\end{proof}

\subsection{Combinatorial proofs}  In this section, we explain, bijectively, the previous formulas for the average number of occurrences of the various subword patterns. We first consider the cases for descents, ascents, and $321$-subwords.\\

\noindent\textbf{Combinatorial proofs of Corollaries \ref{descor1}, \ref{asccor} and \ref{cor321}.}\\

We first treat Corollary \ref{descor1}. Upon multiplying by $n!$, one sees that the total number of descents in the flattened sense within all of the permutations of $[n]$ is given by $\binom{n-1}{2}(n-1)!$ if $n \geq 1$.  This formula may then be explained combinatorially as follows.  First select two members $i<j$ of $[2,n]$.  It is enough to show that the number of permutations $\pi$ of length $n$ such that the letters $i$ and $j$ comprise a descent in $\text{Flatten}(\pi)$ is $(n-1)!$.  To do so, first write the letters in $[i-1]$ as a permutation in standard cycle form.  Then add the letter $i$ to this permutation either within a current cycle (following some member of $[i-1]$) or as a new cycle $(i)$ of length one.  In the former case, we then write the letter $j$ just before $i$ within its cycle, while in the latter case, we write $j$ at the end of the cycle directly preceding the $1$-cycle $(i)$.  Finally, add the letters in $[i+1,n]-\{j\}$ so that no letter comes bet
 ween $j$ and $i$.  Note that there are $(n-1)!$ ways in which to arrange all of the letters in $[n]$ subject to the above restriction upon treating $ji$ as a single letter, which completes the proof of Corollary \ref{descor1}.

For Corollary \ref{asccor}, first note that each of the $n-1$ adjacencies within any member of $\mathcal{S}_n$ is either a flattened ascent or descent.  By subtraction and Corollary \ref{descor1}, there are
$$(n-1)n!-\binom{n-1}{2}(n-1)!=\frac{(n-1)(n+2)}{2}(n-1)!$$
flattened ascents within all of the permutations of length $n$.

A proof similar to the one given for Corollary \ref{descor1} applies to Corollary \ref{cor321} and shows that the total number of occurrence of $321$ is $(n-2)!\binom{n-1}{3}$.  One would now choose three elements $i<j<k$ in $[2,n]$ and treat the string $kji$ as a single letter (equivalent to $i$) when forming a permutation of length $n$ having $kji$ as a subword in the flattened form. Note that such a string may lie completely within a cycle, or straddle two cycles, with $i$ being the first letter of a new cycle.  \hfill \qed\\

\noindent\textbf{Combinatorial proofs of Corollaries \ref{corpeak} and \ref{corva}.}\\

We first treat Corollary \ref{corva} and show equivalently in this case that there are $\frac{n-3}{3}n!$ valleys in the flattened sense within all of the members of $\mathcal{S}_n$, where $n \geq 3$.  We first consider valleys of the form $abc$, where the letter $c$ is \emph{not} the first letter of some cycle within a permutation, while the letter $b$ may or may not be.  Note that there are $2\binom{n-1}{3}$ choices for $a,b,c \in [2,n]$, as there is no restriction on the relative sizes of $a$ and $c$.  Once $a$, $b$, and $c$ have been selected, there are $(n-2)!$ permutations having a valley involving these letters (in the given order), upon treating the string $abc$ as a single letter (equivalent to the letter $b$, in fact).  In all, there are $2\binom{n-1}{3}(n-2)!$ valleys of the given form.

To complete the proof, we must show that there are
$$\frac{n-3}{3}n!-2\binom{n-1}{3}(n-2)!=4\binom{n-1}{3}(n-3)!$$
valleys in the flattened sense of the form $rst$, where $t$ starts a cycle, within all of the members of $\mathcal{S}_n$.  To do so, we first count the number of occurrences of $rs$ within all the permutations of length $n$ (expressed in standard cycle form) such that (i) $r>s$; (ii) either $r$ and $s$ occur in the same cycle as the last two letters in their cycle or $r$ occurs at the end of a cycle, with $(s)$ the next cycle; and (iii) either no cycle follows the one containing $s$ (in either case of (ii)) or the cycle directly following the one containing $s$ starts with a letter that is greater than $s$.

Note that there are $2\binom{n-1}{2}(n-2)!$ occurrences of $rs$ satisfying the conditions (i)-(iii).  To see this, first pick any two elements $r,s \in [2,n]$ and then arrange the remaining members of $[n]$ as a permutation $\sigma$ in standard cycle form in any one of the possible $(n-2)!$ ways.  Once this is done, either add the string $rs$ to the end of the cycle of $\sigma$ whose smallest element is largest among those cycles whose smallest element is less than $s$, or add $r$ to the end of this cycle of $\sigma$ and then add the $1$-cycle $(s)$ directly after it.  In the latter case, note that the cycle $(s)$ might possibly go between two cycles of $\sigma$, but that the ordering of the cycles would be preserved in this case nonetheless.

From all the occurrences of $rs$ satisfying conditions (i)-(iii) above, we subtract those occurrences in which no cycle follows the one containing $s$ in (iii).  (Note that this will give the total number of valleys $rst$, where $t$ starts a cycle.)  To count these occurrences of $rs$, we equivalently count the permutations $\pi$ of $[n]$ containing two letters $r$ and $s$ such that $r>s$, where $r$ and $s$ are the last two letters in $\text{Flatten}(\pi)$.  To count such permutations, first pick three numbers $u<s<r$ of $[n]$.  Arrange the members of $[u-1]$ as a permutation in standard form; then add $u$ as a $1$-cycle to this permutation; then add the members of $[u+1,n]-\{s,r\}$ such that no letter starts a new cycle; finally, either add $r$ and $s$ to the end of the cycle containing $u$ or just add $r$ to the end of this cycle along with the $1$-cycle $(s)$.  Note that there are
$$2(u-1)!\prod_{i=u+1}^{s-1}(i-1)\prod_{j=s+1}^{r-1}(j-2)\prod_{k=r+1}^{n}(k-3)=2(n-3)!$$
ways in which to arrange the members of $[n]$ as described, once $u$, $s$  and $r$ have been chosen.

This implies that there are $2\binom{n}{3}(n-3)!$ permutations of the form described in the previous paragraph and thus the same number of occurrences of $rs$ satisfying (i)-(iii) in which there is no cycle following the one containing $s$.  By subtraction, we get
$$2\binom{n-1}{2}(n-2)!-2\binom{n}{3}(n-3)!=4\binom{n-1}{3}(n-3)!$$
valleys of the form $rst$ where $t$ starts a cycle within all of the members of $\mathcal{S}_n$, as desired, which completes the proof of Corollary \ref{corva}.

For Corollary \ref{corpeak}, first note that there are the same number of peaks as there are valleys in the flattened sense within a permutation $\pi$ if and only if the last two letters of $\text{Flatten}(\pi)$ comprise an ascent and there is one more peak than valley if and only if the last two letters of $\text{Flatten}(\pi)$ comprise a descent. From the proof above for valleys, we see that there are $2\binom{n}{3}(n-3)!=\frac{n!}{3}$ permutations of length $n$ whose last two letters form a flattened descent.  Thus, there are $\frac{n!}{3}$ more peaks than valleys within all of the members of $\mathcal{S}_n$ and so the total number of peaks is $\frac{n-2}{3}n!$, by the prior result. \hfill \qed\\

\noindent\textbf{Combinatorial proof of Corollary \ref{cor123}.}\\

First note that, by subtraction, we have
$$\text{total(123)}=\text{total(ascents)}-\text{total(ascents~at~ end)}-\text{total(peaks)},$$
where $\text{total}(\s)$ denotes the total number occurrences of the statistic in the flattened sense within all of the members of $\mathcal{S}_n$ (by \emph{ascent at end}, we mean an ascent involving the final two letters of the flattened form).  From the proof of Corollary \ref{corva} above, we see that there are $2\binom{n}{3}(n-3)!=\frac{n!}{3}$ flattened descents in all involving the final two letters and hence $n!-\frac{n!}{3}=\frac{2n!}{3}$ ascents in all involving these letters, by subtraction.  By Corollaries \ref{asccor} and \ref{corpeak}, we then have
$$\text{total(123)}=\frac{(n-1)(n+2)}{2}(n-1)!-\frac{2n!}{3}-\frac{n-2}{3}n!=\frac{n^2+3n-6}{6}(n-1)!,$$
which completes the proof. \hfill \qed\\

\section{Conclusion}

Lemma~\ref{lem:reduction} holds for all patterns of length three and can even be generalized to patterns of greater length. In fact, there are other patterns involving three letters. For example,
there are the patterns $132$, $213$, $231$, $312$, and those of the form $(i,j,k)$ where either
$i=\max\{i,j,k\}$ or $k=\max\{i,j,k\}$ or $i=\min\{i,j,k\}$ or $k=\min\{i,j,k\}$.
Among these eight patterns, only the pattern of $i=\min\{i,j,k\}$
admits to the exchanging trick,
which is necessary for solving it by the technique featured in the current paper.
However, for the pattern $i=\min\{i,j,k\}$, we will need other techniques to deal with $g_n(1ij)$ in the case when $j>i$.

We note further that the methods of the third section provide not only explicit formulas for $g^\s_n$ but also formulas for $g^\s_n(1k)$, where $2 \leq k \leq n$ and $\s \in \{\rr,\dd,\pp,\vv\}$, upon substituting the respective expressions for $g^\s_n$ back into Theorems \ref{thm:123-g1k}, \ref{thm:321-g1k}, \ref{thm:peak-gf} and \ref{thm:valley-gf}.

We also remark that the method presented in the previous section
applies to the descent statistic discussed in Section~\ref{sec:des}.
Let
$g^\des_n(a_1a_2\cdots a_k)=\sum_{\pi}q^{\des(\text{Flatten}(\pi))}$,
where~$\pi$ ranges over all permutations
of length~$n$ such that $\text{Flatten}(\pi)$ starts with $a_1a_2\cdots a_k$.
One may show for all $k\ge3$ that
\[
g^\des_n(1k)
=\sum_{i=1}^{k-1}\alpha_{k,i}(-\theta)^{i-1}\cdotp g^\des_{n-i},
\]
where $\alpha_{k,i}={k-1\choose i-1}-{k-3\choose i-3}$ is as in
Theorem~\ref{thm:123-g1k} and $g_n=g_n(1)$. Along these same lines, we may deduce
\[
g^\des_n=\sum_{i=1}^n\Biggl[{n\choose i}-{n-2\choose i-2}\Biggr](-\theta)^{i-1}\cdotp g^\des_{n-i},
\]
which implies the generating function~\eqref{gf:des}.

\noindent{\bf Acknowledgements.}
The third author was supported by the National Natural Science Foundation of China (Grant No. 11101010).

%----------------------------------------------------------------------------------------------------

\end{document}